\documentclass{amsart}

\usepackage{psfrag}
\usepackage{color}
\usepackage{tikz}
\usetikzlibrary{matrix,arrows}
\usepackage{graphicx,graphics}
\usepackage{amssymb,amsfonts,amsmath,amstext,amsthm,amscd,verbatim,enumerate,soul}
\usepackage[T1]{fontenc}

\DeclareMathOperator{\dist}{dist}

\begin{document}

\newtheorem{theorem}{Theorem}[section]
\newtheorem{result}[theorem]{Result}
\newtheorem{fact}[theorem]{Fact}
\newtheorem{conjecture}[theorem]{Conjecture}
\newtheorem{lemma}[theorem]{Lemma}
\newtheorem{proposition}[theorem]{Proposition}
\newtheorem{corollary}[theorem]{Corollary}
\newtheorem{facts}[theorem]{Facts}
\newtheorem{props}[theorem]{Properties}
\newtheorem*{thmA}{Theorem A}
\newtheorem{ex}[theorem]{Example}
\theoremstyle{definition}
\newtheorem{definition}[theorem]{Definition}
\newtheorem*{remark}{Remark}
\newtheorem{example}[theorem]{Example}
\newtheorem*{defna}{Definition}

\newcommand{\notes} {\noindent \textbf{Notes.  }}
\newcommand{\defn} {\noindent \textbf{Definition.  }}
\newcommand{\defns} {\noindent \textbf{Definitions.  }}
\newcommand{\x}{{\bf x}}
\newcommand{\e}{\epsilon}
\renewcommand{\d}{\delta}
\newcommand{\z}{{\bf z}}
\newcommand{\B}{{\bf b}}
\newcommand{\V}{{\bf v}}
\newcommand{\T}{\mathbb{T}}
\newcommand{\Z}{\mathbb{Z}}
\newcommand{\Hp}{\mathbb{H}}
\newcommand{\D}{\Delta}
\newcommand{\R}{\mathbb{R}}
\newcommand{\N}{\mathbb{N}}
\renewcommand{\B}{\mathbb{B}}
\renewcommand{\S}{\mathbb{S}}
\newcommand{\C}{\mathbb{C}}
\newcommand{\ft}{\widetilde{f}}
\newcommand{\rt}{\widetilde{\rho}}
\newcommand{\dt}{{\mathrm{det }\;}}
 \newcommand{\adj}{{\mathrm{adj}\;}}
 \newcommand{\0}{{\bf O}}
 \newcommand{\av}{\arrowvert}
 \newcommand{\zbar}{\overline{z}}
 \newcommand{\xbar}{\overline{X}}
 \newcommand{\htt}{\widetilde{h}}
\newcommand{\ty}{\mathcal{T}}
\newcommand\diam{\operatorname{diam}}
\renewcommand\Re{\operatorname{Re}}
\renewcommand\Im{\operatorname{Im}}
\newcommand{\tr}{\operatorname{Tr}}
\renewcommand{\skew}{\operatorname{skew}}
\newcommand{\vol}{\operatorname{vol}}

\newcommand{\ds}{\displaystyle}
\numberwithin{equation}{section}
%
%
%
%
\newcommand{\M}{{\mathcal{M}}}
\newcommand{\tef}{transcendental entire function}
\newcommand{\qfor}{\quad\text{for }}
\newcommand*{\defeq}{\mathrel{\vcenter{\baselineskip0.5ex \lineskiplimit0pt
 \hbox{\scriptsize.}\hbox{\scriptsize.}}}
 =}
%
%

\renewcommand{\theenumi}{(\roman{enumi})}
\renewcommand{\labelenumi}{\theenumi}

\newcommand{\alastair}[1]{{\scriptsize \color{red}\textbf{Alastair's note:} #1 \color{black}\normalsize}}

\title[On the mean radius]{On the mean radius of quasiconformal mappings}

\author{Alastair N. Fletcher}

\author{Jacob Pratscher}


\date{\today}

\begin{abstract}
We study the mean radius growth function for quasiconformal mappings. We give a new sub-class of quasiconformal mappings in $\R^n$, for $n\geq 2$, called bounded integrable parameterization mappings, or BIP maps for short. These have the property that the restriction of the Zorich transform to each slice has uniformly bounded derivative in $L^{n/(n-1)}$. For BIP maps, the logarithmic transform of the mean radius function is bi-Lipschitz. We then apply our result to BIP maps with simple infinitesimal spaces to show that the asymptotic representation is indeed quasiconformal by showing that its Zorich transform is a bi-Lipschitz map.
\end{abstract}

\maketitle

\section{Introduction}

The growth of quasiconformal mappings is a topic that has been well studied. If we set
\[ L_f(x_0,r) = \max _{|x-x_0| = r} |f(x) - f(x_0) | \]
and
\[ \ell_f(x_0,r) = \min_{|x-x_0| = r} |f(x) - f(x_0)|,\]
then it is well-known that for quasiconformal mappings we have
\[ \limsup_{r\to 0} \frac{ L_f(x_0,r) }{\ell_f(x_0,r) } \leq H \]
for some $H\geq 1$. On the other hand, for a fixed $x_0$, the (real) functions $L_f$ and $\ell_f$ are bounded above and below via H\"older-type inequalities. 
The local distortion result given by \cite[Theorem III.4.7]{Rickman} shows that if $f$ is $K$-quasiconformal then there exist positive constants $A,B$ and and $r_0>0$ so that
\[ A r^{K^{1/(n-1)}} \leq \ell_f(x_0,r) \leq L_f(x_0,r) \leq B r^{K^{1/(1-n)}} \]
for $0<r<r_0$.

Another quantity that is comparable to $L_f$ and $\ell_f$ on small scales is the mean radius function that, as far as the authors are aware, was introduced into the study of quasiregular mappings by Gutlyanskii et al in \cite{GMRV}. We recall this notion here.
Let $n\geq 2$, suppose that $U\subset \R^n$ is a domain and $f:U \to \R^n$ is a non-constant quasiregular mapping. Let $x_0 \in U$ and $r_0 = \operatorname{dist}(x_0, \partial U)$. For $0<r<r_0$, the {\it mean radius} of the image of the ball $B(x_0,r) = \{ x\in \R^n : |x-x_0| <r \}$ under $f$ is defined by
\[ \rho _f(x_0,r) = \left ( \frac{ \vol_n f(B(x_0 , r) )  }{ \Omega_n} \right )^{1/n}, \]
where $\vol_n E$ denotes the $n$-dimensional Lebesgue measure of $E$ and $\Omega_n$ denotes the volume of the unit ball $\B^n = B(0,1) \subset \R^n$.

Since
\[ \Omega_n \ell_f(x_0,r) \leq \vol_n f(B(x_0,r) ) \leq \Omega_n L_f(x_0,r),\]
it is clear that
\[ \ell_f(x_0,r) \leq \rho_f(x_0,r) \leq L_f(x_0,r).\]
Henceforth, we will assume that $x_0=0$. This makes no difference to our discussion.

The mean radius was used by Gutlyanskii et al \cite{GMRV} to construct generalized derivatives of quasiregular mappings, a topic which has been studied more recently in papers of the first named author and Wallis \cite{FW} and the authors \cite{FP}. In this paper, we will study the regularity properties of $\rho_f$ itself. It is clear from the definition that $\rho_f$ is increasing and, moreover, it was shown in the proof of \cite[Lemma 3.1]{FW} that $\rho_f$ is continuous as a consequence of the Lusin (N) condition. It follows from a standard result in analysis that, as a continuous monotonic real function, $\rho_f$ is differentiable almost everywhere. 

It will be convenient for our results to use the logarithmic transform of $\rho_f$. Suppose $\rho_f :(0,r_0) \to (0,\infty)$. Then the {\it logarithmic transform} of $\rho_f$ is defined by the function $\widetilde{\rho_f} : (-\infty , \ln r_0) \to \R$ given by
\[ \widetilde{\rho_f} (t) = \ln \rho_f (e^t).\]
We will typically use the $r$-variable for real functions, and the $t$ variable for logarithmic transforms.

As a generalization of the logarithmic transform to quasiregular maps, we will also use the Zorich transform introduced by the authors in \cite{FP}. This will be described in further detail below, but for a quasiconformal map $f$ fixing $0$, this is defined via a particular Zorich map $\mathcal{Z}$ and the functional equation $\mathcal{Z} \circ \widetilde{f} = f\circ \mathcal{Z}$. The domain of definition of $\widetilde{f}$ is a half-beam $B_M = Q \times (-\infty, M)$, where $\overline{Q}$ is a cuboid in $\R^{n-1}$ of the form $[0,1]^{n-2} \times [0,2]$.

To illustrate why the logarithmic transform is a natural object to study in the growth of quasiconformal mappings, consider the following example.

\begin{example}
\label{ex:1}
Let $f:\C \to \C$ be a radial quasiconformal mapping of the form $f(re^{i\theta}) = h(r) e^{i\theta}$ in polar coordinates, where $h:[0,\infty) \to [0,\infty )$ is a strictly increasing homeomorphism satisfying $h(0)=0$. Then via a direct computation we see that its complex dilatation is
\[ \mu_f = \frac{ f_r + \frac{i}{r} f_{\theta} }{ f_r - \frac{i}{r} f_{\theta} } = \frac{ \frac{rh'(r) }{h(r)} - 1 }{ \frac{rh'(r) }{h(r)} + 1 },\]
wherever $h'(r)$ exists. Since $h$ is an increasing homeomorphism, it is differentiable almost everywhere, and so $\mu_f$ has this expression for almost every $r$ value. Letting $\widetilde{h}$ be the logarithmic transform of $h$, and as
\[ \widetilde{h}'(t) = \frac{e^t h'(e^t) }{h(e^t)},\]
it is then clear that for $f$ to be quasiconformal, it is necessary that there exist constants $0<C_1<C_2 <\infty$ so that
\[ C_1 \leq \widetilde{h} '(t) \leq C_2\]
for all $t$ where $\widetilde{h}'$ is defined.
\end{example}

Radial maps are the nicest class of mappings from our viewpoint, since $\ell_f, L_f$ and $\rho_f$ will all coincide. In general, we expect other factors to play a role in the behaviour of $\rho_f$.
To that end, we build towards the key definition of our paper.

Let $n\geq 2$. For vectors $v_1,\ldots, v_{n-1}$ in $\R^n$, form the $n$ by $n$ matrix $A$ where the top row consists of the unit vectors $e_1,\ldots, e_n$ in $\R^n$, and the $i$th row of $A$ consists of entries from the vector $v_{i-1}$, for $2\leq i \leq n$. Then denote by $\Pi (v_1,\ldots, v_{n-1} )$ the vector in $\R^n$ given by $\det A$.

Returning to the Zorich transform, $\ft$ maps each slice $Q\times \{ t \}$, for $-\infty <t < M$, onto a hypersurface in the beam $B$. The map $\ft$ provides a parameterization for each image hypersurface through $Q$. Recalling that 
\[ \overline{Q} = \{ (x_1,\ldots, x_{n-1}) : 0\leq x_1,\ldots, x_{n-2} \leq 1, 0\leq x_{n-1} \leq 2 \},\]
for $-\infty <t<M$, write $\gamma :Q \to \ft (Q\times \{ t \} )$ so that 
\[ \gamma(x_1,\ldots, x_{n-1}) = \ft ( x_1,\ldots, x_{n-1}, t).\]
If $\gamma(Q)$ has a well-defined $(n-1)$-dimensional volume in $B$ (noting that this is not necessarily the case for an arbitrary quasiconformal Zorich transform $\ft$), then it is given by
\[ \vol_{n-1} \gamma(Q) = \int_Q \left | \left | \Pi \left ( \frac{ \partial \gamma}{\partial x_1 } , \ldots, \frac{ \partial \gamma}{\partial x_{n-1} } \right ) \right | \right | _2 dV ,\]
where $dV$ denotes the volume element on $Q$, and $||.||_2$ denotes the usual $2$-norm.

\begin{definition}
\label{def:main}
Let $n\geq 2$, let $e^M>0$ and let $f:B(0,e^M) \to \R^n$ be quasiconformal with $f(0) = 0$. We say that $f$ h is a {\it bounded integrable parameterization map}, or is a {\it BIP map}, if there exists $P>0$ such that for every $t<M$, the parameterization $\gamma :Q \to \ft ( Q \times \{t \})$ via $\ft$ satisfies
\[ \int_Q \left | \left | \Pi \left ( \frac{ \partial \gamma}{\partial x_1 } , \ldots, \frac{ \partial \gamma}{\partial x_{n-1} } \right ) \right | \right | ^{n/(n-1)}_2 dV \leq P .\]
\end{definition}

Since $n/(n-1) >1$, it follows that the BIP condition implies the images of the various slices $\ft (Q \times \{ t \} )$ have finite $(n-1)$-dimensional volume, but it also implies stronger regularity of the parameterizations.

For our main result, we show that BIP mappings have improved regularity for $\widetilde{\rho_f}$ than just almost everywhere differentiability.

\begin{theorem}
\label{thm:1}
Let $n\geq 2$, $K\geq 1$,  $e^M>0$ and let $f:B(0,e^M) \to \R^n$ be a BIP $K$-quasiconformal map with $f(0) = 0$. Then there exists $L$ depending only on $n$, $K$ and $P$ such that $\widetilde{\rho_f}$ is $L$-bi-Lipschitz on $(-\infty, M)$.
\end{theorem}

It follows from Theorem \ref{thm:1} that 
\[ \frac{1}{L} \leq \widetilde{\rho_f}'(t) \leq L \]
for every $t\in (-\infty, M)$ at which $\widetilde{\rho_f}$ is differentiable, recalling Example \ref{ex:1}.
As will be evident from the proof of Theorem \ref{thm:1}, the implied lower bound in the bi-Lipschitz behaviour of $\widetilde{\rho_f}$ always holds. The requirement of BIP maps is needed for the upper bound. 

Theorem \ref{thm:1} may be applied in neighbourhoods of a fixed point of a quasiregular mapping, provided that the local index $i(x_0,f) =1$, that is, that $f$ is injective in a neighbourhood of $x_0$. If the fixed point of $f$ is geometrically attracting, then by definition there exist $c>1$ and $r_1 > 0$ such that $|f(x)| < c|x|$ for $|x| < r_1$. It is worth pointing out that although $\rho_f (r) <c r$, it does not follow that $\widetilde{\rho_f} '(t) <1$: consider the simple example $f(x) = x/2$ for which $\widetilde{\rho_f}(t) = t - \ln 2$.

The idea behind the proof of Theorem \ref{thm:1} is as follows. We first need to show that the Zorich transform $\ft$ is not just quasiconformal in $B_M$, but in fact quasisymmetric. This allows us to compare quantities of the form $(\widetilde{\rho_f}(t_0+t) - \widetilde{\rho_f}(t_0) )/t$ to $n$-dimensional volumes of images of slices $Q \times [t_0,t]$ under $\ft$. Subdividing these slices into small cuboids and using the quasisymmetry of $\ft$ allows us to compare the volume of the image of a cuboid to its diameter raised to the $n$'th power. Summing over all such cuboids and a judicious use of an $\ell^p$ inequality yields the lower bound. The upper bound arises by using the sum over all cuboids as an approximation for the expression in Definition \ref{def:main}. The BIP condition is then needed to provide the upper bound.

We record the fact that the BIP condition is necessary for Theorem \ref{thm:1} in the following proposition.

\begin{proposition}
\label{prop:0}
Suppose that $e^M >0$ and let $f:B(0,e^M) \to \R^2$ be quasiconformal, fix $0$ and so that for some $0<e^{t_0}<e^M$, the image of the circle centred at $0$ of radius $e^{t_0}$ under $f$ is a non-rectifiable curve. Then $\widetilde{\rho_f}$ is not bi-Lipschitz at $t_0$.
\end{proposition}

It is well-known that the image of a circle under a quasiconformal map need not be rectifiable, for example, it could be a snowflake curve in $\R^2$ or a snowball in $\R^3$. We refer to \cite{Meyer} for snowballs and references in the literature addressing snowflake curves.

As an application of Theorem \ref{thm:1}, we turn to infinitesimal spaces. Recall from \cite{GMRV} that a generalized derivative of a quasiregular map $f:U\to \R^n$ at $x_0 \in U$ is defined to be any local uniform limit of 
\[ \frac{ f(x_0 + r_k x) - f(x_0) }{ \rho_f(r_k) } \]
as $r_k \to 0$. Of course, not every such sequence need have a limit, but the quasiregular version of Montel's Theorem implies there will be a subsequence along which there is local uniform convergence. See \cite[p.103]{GMRV} for a discussion of this point. The collection of generalized derivatives of $f$ at $x_0$ is called the {\it infinitesimal space} $T(x_0,f)$. 

In the special case where $T(x_0,f)$ consists of only one mapping $g$, then $f$ is called {\it simple} at $x_0$. For example, if $f$ is differentiable at $x_0$, then $f$ is simple at $x_0$. If $f$ is simple at $x_0$, then $f$ has an asymptotic representation analogous to a first degree Taylor polynomial approximation of an analytic function. As usual, we suppose $x_0 = f(x_0) = 0$. Then \cite[Proposition 4.7]{GMRV} states that as $x\to 0$, we have
\[ f(x) \sim D(x) := \rho_f(|x|) g(x/ |x| ),\]
where $p(x) \sim q(x)$ as $x\to 0$ means
\[ |p(x) - q(x) | = o( |p(x)| + |q(x) | ).\]
The map $D$ is called the {\it asymptotic representative} of $f$ at (in this case) $0$.
Our second main result concerns the asymptotic representative.

\begin{theorem}
\label{thm:2}
Let $n\geq 2$, $U\subset \R^n$ be a domain and
suppose that $f:U \to \R^n$ is a quasiregular map, $0 \in U$, $i(0,f) = 1$ and $f$ is simple at $0$ and $f$ is a BIP quasiconformal map on some neighbourhood of $0$. Then there exists $M\in \R$ such that $\widetilde{D}$ is bi-Lipschitz in $B_M$. Moreover, $D$ is quasiconformal in $B(x_0,e^M)$.
\end{theorem}

The simplicity assumption on $f$ implies via \cite[Proposition 4.7]{GMRV} that $\rho_f$ is asymptotically $d$-homogeneous, that is, there exists $d>0$ such that for any $s>0$,
\[ \rho_f( st) \sim s^d\rho_f(t) \]
as $t\to 0$. It should be noted that this conclusion does not imply that $\rho_f$ itself is $d$-homogeneous: one can check that if $d>0$ then the two dimensional radial map given in polar coordinates by
\[ (r,\theta ) \mapsto \left ( \frac{r^d}{\log (1/r)} , \theta \right )\]
is quasiconformal, but $\rho_f$ is not $d$-homogeneous.

The assumption that $f$ is BIP in a neighbourhood of $0$ is necessary in Theorem \ref{thm:2}. We illustrate this by means of an example.

\begin{example}
We work in dimension two.
There exists a quasiconformal map $h$ from the square $[-1,1]^2$ to itself which is the identity on the boundary and so that $h$ maps $[-1,1]\times \{ 0 \}$ onto a non-rectifiable curve.
Next, for $M<0$, let $B_M = [0,2] \times (-\infty, M)$.
For $d\in \Z$ and $d<\min \{M, -2 \}$, consider the square $S_d$ centred at $z_d = (1,d)$ with sides parallel to the coordinate axes and of side-length $|d|^{-1}$. We define $\ft$ to be the identity outside $S_d$ in $B_M$. Let $A_d$ be the linear map
\[ A_d(z) = \frac{z}{2|d|} + z_d\]
and then for $z\in S_d$, we define $\ft(z) = A_d(h(A_d^{-1}(z)))$.

There exists a quasiconformal map $f: B(0,e^M) \to \R^2$ whose Zorich transform is $\ft$. By construction, for $r>0$, $f$ is not the identity in $B(0,r)$ on the union of the sets $\mathcal{Z}(S_d)$, where $d < \log r$. For such $d$, the diameter of $\mathcal{Z}(S_d)$ is bounded above by
\[ e^{d +1/2|d|} -e^{d-1/2|d|} < 2\sinh \left ( \frac{1}{2\ln \frac{1}{r} } \right ) = o(1) \]
as $r\to 0$. We conclude that $f(x) \sim x$ as $x\to 0$ and so $f$ has a simple infinitesimal space consisting of the identity. However, $\widetilde{\rho_f}$ is not bi-Lipschitz on any neighbourhood of $d$ by the argument in Proposition \ref{prop:0}, from which it follows that $D(x) = \rho_f(|x|) x/|x|$ is not quasiconformal.
\end{example}

As a final remark, in \cite{FM}, it was implicitly assumed that the map $D$ is quasiconformal. Since our results above show that this is not always the case, Theorem \ref{thm:2} can be viewed as providing an extra assumption necessary for the results from that paper to hold.

\section{Preliminaries}

Throughout, $B(x_0,r)$ denotes the ball centred at $x_0 \in \R^n$ of radius $r>0$. Its boundary is the sphere $S(x_0,r)$ or, if $x_0=0$, we may simply write $S(r)$. If $x\in \R^n$, then we write $[x]_n \in \R$ for the $n$th coordinate.

We refer to standard references such as \cite{Rickman} for the definition and basic properties of quasiregular and quasiconformal mappings in $\R^n$. 
In particular, if $f$ is quasiregular, then $K_O(f), K_I(f)$ and $K(f)$ refer to the outer dilatation, inner dilatation and maximal dilatation, respectively.
Bounded length distortion maps, BLD for short, are a sub-class of quasiregular maps for which the finite length curves are mapped to curves of finite length, with uniform control on the length distortion. BLD maps are locally bi-Lipschitz. In some sense, BIP maps provide a generalization of BLD maps.

We will use the following result.

\begin{theorem}[Theorem 1.1, \cite{FN}]
\label{thm:fn}
Let $U\subset \R^n$ be a domain for $n\geq 2$ and let $f:U \to \R^n$ be a non-constant quasiregular map. If $x\in U$, there exist $C>1$ and $r_0>0$ such that for all $0<T \leq 1$ and $r\in (0,r_0)$
\[ T^{-\mu} \leq \frac{L_f(x,r)}{\ell_f(x,Tr)} \leq C^2 T^{-\nu}\]
and
\[ \frac{T^{-\mu}}{C^2} \leq \frac{ \ell_f(x,r) }{L_f( x,Tr) } \leq T^{-\nu},\]
where $\mu = (i(x,f) / K_I(f) )^{1/(n-1)}$, $\nu = (K_O(f) i(x,f) )^{1/(n-1)}$ and $C$ depends only on $n$, $K_O(f)$ and the local index $i(x,f)$.
\end{theorem}

\subsection{The Zorich transform}

The class of Zorich mappings provides a well-known quasiregular generalization of the exponential function in the plane. 
These maps are strongly automorphic with respect to a discrete group $G$ of isometries. This means that if $\mathcal{Z}$ is a Zorich map, then $\mathcal{Z}(g(x)) = \mathcal{Z}(x)$ for all $g\in G $ and all $x\in \R^n$ and, moreover, if $\mathcal{Z}(x) = \mathcal{Z}(y)$ then $y=g(x)$ for some $g\in G$.

For each $n \geq 2$, we will fix one particular Zorich map, which we denote simply by $\mathcal{Z}$, with the following properties. 
\begin{enumerate}[(i)]
\item For each $t \in \R$,  $\mathcal{Z}$ maps the hyperplane $H_t = \{ x\in \R^n : [x]_n =t \}$ onto the sphere $S(e^t)$.\label{spheres}
\item On each hyperplane $H_t$, $\mathcal{Z}$ is locally bi-Lipschitz with isometric distortion scaling with $e^t$. In fact,
\[ \mathcal{Z}(x_1,\ldots, x_{n-1},t) =e^t \mathcal{Z}(z_1,\ldots, x_{n-1},0).\] 
\item The group $G$ is generated by a translation subgroup of rank $n-1$ and a finite group of rotations about the origin such that each $g\in G$ preserves the $n$'th coordinate.
\item A fundamental set $B$ for the action of $G$ is given by a beam $Q \times \R$, where the closure of $Q$ is a cuboid in $\R^{n-1}$ of the form $[0,1]^{n-2} \times [0,2]$ and the sides of this beam are identified via the group $G$.
\end{enumerate}

For an explicit formula for such a Zorich map, together with its associated group $G$, we refer to, for example, \cite[Section 3]{FP}.
If $f$ is quasiconformal in a neighbourhood $U$ of the origin in $\R^n$ and $f(0)=0$, then we may consider the {\it Zorich transform} $\widetilde{f}:\mathcal{Z}^{-1}(U) \cap B \to B$ of $f$. 
This is defined via the relation
\[ f\circ \mathcal{Z} = \mathcal{Z} \circ \widetilde{f}.\]
We often just assume the domain of $\widetilde{f}$ is 
\[ B_M := \{ x\in B : [x]_n <M \} , \] 
in the same way that we may assume that the domain of $f$ is $\{ x : |x| <e^M \}$.  

We will use the fact that $B_M$ can be made into a metric space via the quotient of the Euclidean metric under the group $G$ and consider the appropriate topological notions of balls, convergence and so on in this metric space. In particular, $\mathcal{Z}^{-1} :B_M \to \R^n$ is locally bi-Lipschitz with this metric in the domain and the Euclidean metric in the range.

\subsection{Quasisymmetric maps}

A standard reference for quasisymmetric maps is Heinonen's book \cite{Heinonen}. A map $f:X\to Y$ between metric spaces is called {\it quasisymmetric} if there exists a homeomorphism $\eta :[0,\infty) \to [0,\infty)$ such that for all triples of points $x,a,b \in X$, we have
\[ \frac{ d_Y( f(x), f(a) )}{d_Y( f(x), f(b) )} \leq \eta \left ( \frac{ d_X(x,a) }{d_X(x,b) } \right).\]
If we wish to specify the function $\eta$, then $f$ is said to be {\it $\eta$-quasisymmetric}. A map $f:X\to Y$ is called {\it weakly quasisymmetric} if there exists a constant $H\geq 1$ such that $d_X(x,a) \leq d_X(x,b)$ implies 
\[ d_Y(f(x),f(a)) \leq H d_Y(f(x),f(b) )\]
for all triples of points $x,a,b \in X$. While in general weakly quasisymmetric maps need not be quasisymmetric, they are in connected doubling spaces, see \cite[Theorem 10.19]{Heinonen}.

Quasisymmetry can be viewed as a global version of quasiconformality. In particular, a quasisymmetric map $f:U\to V$ between domains in $\R^n$ is quasiconformal. Conversely, the egg-yolk principle states that a $K$-quasiconformal map $f:U\to V$ is $\eta$-quasisymmetric on each ball $B(x_0 , \operatorname{dist}(x_0 , \partial U) /2 )$, with $\eta$ depending only on $n$ and $K$, see \cite[Theorem 11.14]{Heinonen}.

\subsection{$\ell^p$ estimates}

It is well-known that if $x = (x_1,\ldots, x_N)\in \R^N$ and $x_i \geq 0$ for $i=1,\ldots, N$, then for $1\leq p <q \leq \infty$ we have
\[
||x||_q \leq ||x||_p \leq N^{1/p - 1/q} ||x||_q,
\]
where $||x||_p$ denotes the usual $p$-norm
\[ ||x||_p = \left ( \sum_{i=1}^N x_i^p \right )^{1/p}.\]
In particular, for $n\geq 2$, if $p=n-1 $ and $q=n$, then
\begin{equation}
\label{eq:lp}
\sum_{i=1}^N x_i^n \geq  N^{1/(1-n)} \left ( \sum_{i=1}^N x_i^{n-1} \right ) ^{n/(n-1)}.
\end{equation}

\section{Quasisymmetry}

We will need the following result that shows the Zorich transform of a quasiconformal map is quasisymmetric. This is a non-trivial result, as $\mathcal{Z}^{-1}$ itself is not quasisymmetric on any punctured neighbourhood of the origin.

\begin{proposition}
\label{prop:1}
Let $n\geq 2$, $K>1$ and $r_0 >0$. If $f :B(0,r_0) \to \R^n$ is a $K$-quasiconformal map with $f(0)=0$, then there exists $M < \ln r_0$ such that $\widetilde{f}$ is quasisymmetric in $B_M$.
\end{proposition}

\begin{proof}
First, for any $M$, as $B_M$ is a connected doubling space, \cite[Theorem 10.19]{Heinonen} implies that it is sufficient to show that $\widetilde{f}$ is weakly quasisymetric. Suppose for a contradiction that $\widetilde{f}$ is not weakly quasisymmetric. Then there exists a sequence of triples of points $x_k,a_k,b_k$ in $B_M$ such that
\begin{equation} 
\label{eq:prop1}
|x_k - a_k | \leq |x_k - b_k |,\quad \text{ but } \quad |\widetilde{f}(x_k) - \widetilde{f}(a_k) | > k |\widetilde{f}(x_k) - \widetilde{f}(b_k) |.
\end{equation}
We will deal with large scales and small scales separately.

First, if $|\widetilde{f}(x_k) - \widetilde{f}(b_k)|$ does not converge to $0$ then, passing to a subsequence if necessary, we have $|\widetilde{f}(x_k) - \widetilde{f}(a_k) | \to \infty$. This implies that $| [ \widetilde{f}(x_k) ]_n - [\widetilde{f}(a_k) ]_n | \to \infty$.
Now set $t_k  = |x_k - b_k|$, so $t_k \to \infty$ by \eqref{eq:prop1}. Without loss of generality, suppose $[a_k]_n \leq [x_k]_n \leq [b_k]_n$. 
Then we have
\begin{equation}
\label{eq:prop1eq1}
[b_k ] _n \leq [x_k]_n +t_k, \quad [a_k]_n \geq [x_k]_n - t_k.
\end{equation}
Next, set $s_k = [b_k]_n - [x_k]_n$. As $t_k^2 \leq s_k^2 + (\diam Q)^2$ by Pythagoras and $(\diam Q)^2 = n+2$, we have
\begin{equation}
\label{eq:prop1eq1a}
\sqrt{ 1 - \frac{ (\diam Q)^2}{t_k^2} } \leq \frac{s_k}{t_k} \leq 1
\end{equation}
and in particular, $s_k/t_k \to 1$ as $k\to \infty$.

Next, set $y_k = \mathcal{Z}(x_k)$, $p_k = \mathcal{Z}(a_k)$ and $q_k = \mathcal{Z}(b_k)$. Using Theorem \ref{thm:fn} and \eqref{eq:prop1eq1}, we have
\begin{align*}
|\ft (x_k) - \ft(a_k) | &= \left | \log \frac{ f(y_k) }{f(p_k)} \right | \\
&\leq \log \frac{ L_f(0,e^{[x_k]_n} ) }{ \ell_f(0,e^{[a_k]_n} )} \\
&\leq \log \frac{ L_f(0,e^{[x_k]_n} ) }{ \ell_f(0, e^{-t_k} e^{[x_k]_n} ) } \\
&\leq \log \left ( C^2 (e^{-t_k} )^{-\nu } \right ),
\end{align*}
recalling that $\nu = K_O(f) ^{1/(n-1)}$ and $C$ only depends on $n$ and $K_O(f)$ since $f$ is quasiconformal. On the other hand, by Theorem \ref{thm:fn} and the definition of $s_k$ we have
\begin{align*}
|\ft (b_k) - \ft(x_k) |&= \left | \log \frac{ f(q_k)}{f(y_k)}  \right | \\
&\geq \log  \frac{ \ell_f(0, e^{[b_k]_n} ) }{ L_f( 0, e^{[x_k]_n} ) } \\
&= \log \frac{ \ell_f(0, e^{[b_k]_n} ) } { L_f (0, e^{-s_k} e^{[b_k]_n} ) } \\
& \geq \log \left (  \frac{ (e^{-s_k} )^{-\mu } }{C^2} \right ),
\end{align*}
where  $\mu = K_I(f) ^{1/(1-n)}$. Hence by \eqref{eq:prop1eq1a} we have 
\begin{align*}
\frac{ | \ft(b_k) - \ft (x_k) |} { |\ft(x_k) - \ft (a_k) | } & \geq \frac{ -2\log C + \mu s_k }{2\log C + \nu t_k} \to \frac{ \nu}{\mu} = (K_O(f)K_I(f))^{1/(1-n)}
\end{align*}
as $k\to \infty$. This provides a contradiction to \eqref{eq:prop1}.

Turning now to the small scales case, assume that $| \widetilde{f}(x_k) - \widetilde{f}(b_k) | \to 0$ and hence $t_k = |x_k-b_k|\to 0$. As $f$ is $K$-quasiconformal, $\widetilde{f}$ is $K'$-quasiconformal, where $K' \leq K[K(\mathcal{Z})]^2$. We may define a sequence of $K'$-quasiconformal maps $g_k : \overline{\B^n} \to \overline{\B^n}$ via
\[ g_k(x) = \frac{ \widetilde{f} ( x_k + |x_k - b_k| x ) - \widetilde{f}(x_k) }{L_{\widetilde{f}} (x_k, |x_k - b_k | ) } .\]
By construction, $\sup_{|x|=1} |g_k(x) | =1$ for all $k$. Since the family $\{g_k \}$ is normal by the quasiregular version of Montel's Theorem (see \cite[Theorem 4]{Miniowitz}), it follows that there exists $\delta >0$ such that 
\begin{equation}
\label{eq:prop2}
\inf_{|x|=1} |g_k(x)| >\delta
\end{equation} 
for all $k$. 

However, as $|x_k - a_k| \leq |x_k - b_k |$ and $ |\widetilde{f}(x_k) - \widetilde{f}(a_k) | > k |\widetilde{f}(x_k) - \widetilde{f}(b_k) |$, it follows that, with $y_k = (b_k - x_k) / |b_k - x_k|$, we have 
\begin{align*}
|g_k(y_k)| &=  \frac{ \widetilde{f} ( b_k ) - \widetilde{f}(x_k) }{L_{\widetilde{f}} (x_k, |x_k - b_k | ) } \\
&\leq \frac{ | \ft(a_k) - \ft(x_k) | }{k L_{\widetilde{f}} (x_k, |x_k - b_k | ) } \\
&\leq  \frac{ L_{\ft} (|x_k - a_k|) }{k L_{\ft} (|x_k - b_k |)} \to 0
\end{align*}
as $k\to \infty$. This is a contradiction to \eqref{eq:prop2}.

We conclude that $\widetilde{f}$ is weakly quasisymmetric, and hence quasisymmetric, on $B_M$.
\end{proof}

We will assume henceforth that $\ft$ is defined in $B_M$, and identify $B_M$ with $Q \times (-\infty, M)$. Given $t_0 < M-1$ and $t<1$, consider the slice 
\[ S_t = Q \times [t_0 , t_0+t ] \subset B_M.\] 
We subdivide $S_t$ into boxes in the following way. Subdivide $Q$ into $(n-1)$-dimensional cubes $T_i$ of side length $( \lfloor 1/t \rfloor )^{-1}$ with edges parallel to the unit vectors $e_1,\ldots, e_{n-1}$. Finally, we set 
\begin{equation}
\label{eq:Pi}
P_i = T_i \times [t_0,t_0+t]
\end{equation}
for $i=1,\ldots, N$.

We collect important geometric information about the quasisymmetric images of $P_i$ in the following lemma. Recall that we denote by $\vol_p$ the $p$-dimensional volume.

\begin{lemma}
\label{lem:qsp}
Let $\ft :B_M \to B$ be $\eta$-quasisymmetric and let $t\in (0,1/2)$.
Then the number of boxes in the subdivision of $S_t$ is $N$, where $N/t^{1-n}\to 2$ as $t\to 0$. Moreover, there exists a constant $C_1 >1$ depending only on $n$ and $\eta$ so that
\[ \frac{1}{C_1} \leq \frac{ \vol_n \ft (P_i) }{ (\diam \ft (P_i)) ^n} \leq C_1 ,\]
for each box $P_i$ constructed as above.
\end{lemma}

\begin{proof}
The process of subdividing $Q$ into $(n-1)$-dimensional cubes of side length $( \lfloor 1/t \rfloor )^{-1}$ yields $N = O(t^{1-n})$ of them as
\begin{equation}
\label{eq:qsp0} 
t\leq ( \lfloor 1/t \rfloor )^{-1} \leq \frac{t}{1-t} \leq 2t,
\end{equation}
$\vol_{n-1} Q =2$ by construction and $\vol_{n-1} Q = N ( \lfloor 1/t \rfloor )^{1-n}$. Note that as $t\to 0$, $(\lfloor 1/t \rfloor) ^{-1} / t \to 1$ and so $N/t^{(1-n)} \to 2$.

Next, let $x_0$ be the centroid of $P_i$. By \eqref{eq:qsp0} we have
\begin{equation}
\label{eq:qsp1}
\frac{t}{2} \leq  \operatorname{dist}( x_0 , \partial P_i) .
\end{equation}
Let $y,z$ be any two points on $\partial P_i$. Then since $\ft$ is $\eta$-quasisymmetric, we have
\[ \frac{ |\ft (y) - \ft (z) | }{ |\ft (y) -\ft (x_0) | } \leq \eta \left ( \frac{ |y-z| }{|y-x_0|} \right ).\]
Choosing $y,z$ so that $|\ft (y) - \ft(z)|$ realizes the diameter of $\ft (P_i)$, we obtain via \eqref{eq:qsp1} that
\begin{equation}
\label{eq:qsp2} 
\operatorname{diam} \ft (P_i) \leq \dist ( \ft(x_0) , \partial \ft (P_i) ) \cdot \eta \left ( \frac{t \sqrt{n} }{  t/2 } \right ) =\dist ( \ft(x_0) , \partial \ft (P_i) ) \cdot  \eta (2\sqrt{n} ) .
\end{equation}
Since 
\[ B(\ft ( x_0) , \operatorname{dist}( \ft (x_0) , \partial \ft (P_i) ) ) \subset \ft (P_i) \subset B( \ft (x_0) , \operatorname{diam} \ft (P_i) ),\]
it follows that
\[ \Omega_n \left ( \operatorname{dist}( \ft (x_0) , \partial \ft (P_i) ) \right )^n \leq | \ft (P_i) | \leq \Omega_n \left ( \operatorname{diam} \ft (P_i)  \right )^n.\]
Combining this with \eqref{eq:qsp2}, we obtain
\[ \frac{\Omega_n }{\eta ( 2\sqrt{n} )} \leq \frac{ |\widetilde{f}(P_i) |}{ ( \operatorname{diam} \widetilde{f}(P_i)  )^n } \leq \Omega_n ,\]
as required.
\end{proof}

\begin{definition}
Given a box $P_i$ constructed as above, denote by $\nu( \ft (P_i))$ the minimum distance between images under $\ft$ of opposite pairs of faces of $P_i$.
\end{definition}

\begin{lemma}
\label{lem:2}
Let $\ft :B_M \to B$ be $\eta$-quasisymmetric and let $t\in (0,1/2)$. There exists a constant $C_2>1$ depending only on $n$ and $\eta$ such that if $P_i$ is a box constructed as in \eqref{eq:Pi} with parameter $t$, then
\[ 1 \leq \frac{ \operatorname{diam}(\ft (P_i)) }{ \nu ( \ft (P_i)) } \leq C_2.\]
\end{lemma}

\begin{proof}
The lower bound here is trivial.
Given $P_i$, find $x,y \in P_i$ which realize the minimum in the definition of $\nu ( \ft (P_i))$. Necessarily, $x$ and $y$ must lie in opposite faces of the boundary of $P_i$. Applying the quasisymmetry of $\ft$ twice, for any $p,q \in P_i$ we have
\[ \frac{ |\ft(p) - \ft(q) | }{ |\ft(x) - \ft(y) | } = \frac{|\ft(p) - \ft(q) |}{|\ft(x) - \ft(q)|} \cdot \frac{|\ft(x) - \ft(q) |}{|\ft(x)-\ft(y)|} \leq \eta \left ( \frac{ |p-q| }{|x-q| } \right ) \eta \left ( \frac{ |x-q|}{ |x-y |} \right ).\]
We choose $p,q$ so that their images realize $\operatorname{diam} \ft (P_i)$ and, without loss of generality, we may assume that $|x-q| \geq t/2$ (if not, switch the roles of $x$ and $y$). From \eqref{eq:qsp0}, we have $|x-q| \leq 2t\sqrt{n}$, $|p-q| \leq 2t \sqrt{n} $ and $|x-y| \geq t$, from which it follows that 
\[  \frac{ |\ft(p) - \ft(q) | }{ |\ft(x) - \ft(y) | } \leq \eta ( 4\sqrt{n} ) \eta ( 2\sqrt{n} ) ,\]
completing the proof.
\end{proof}

\begin{lemma}
\label{lem:3}
Let $\ft :B_M \to B$ be $\eta$-quasisymmetric and let $t\in (0,1/2)$. There exists a constant $C_3$ depending only on $n$ and $\eta$ such that if $P_i$ is a box constructed as in \eqref{eq:Pi} with parameter $t$, and $Q_i$ denotes the base of the box, that is, $Q_i = P_i \cap (Q \times \{t_0 \})$, then
\[ \frac{ \nu (\ft (P_i) ) ^{n-1} }{\vol_{n-1} \ft (Q_i) } \leq C_3.\]
\end{lemma}

\begin{proof}
Given $Q_i$, let $x_0$ be its centroid. As $Q_i$ is an $(n-1)$-dimensional cube of side length $( \lfloor 1/t \rfloor )^{-1}$, it follows from \eqref{eq:qsp0} that
\begin{equation}
\label{eq:l3eq1} 
\frac{t}{2} \leq \operatorname{dist} ( x_0 , \partial Q_i ) \leq t\sqrt{n-1} .
\end{equation}
Here, we denote by $\partial Q_i$ the $(n-2)$-dimensional faces on the edge of $Q_i$.
Let $y,z \in \partial Q_i$ be such that $\ft(y)$ and $\ft(z)$ realize the minimum distance between images under $\ft$ of opposite pairs of faces of $\partial Q_i$. Then $| \ft(y) - \ft(z)| \geq \nu (\ft (P_i))$. By the quasisymmetry of $\ft$ and \eqref{eq:l3eq1}, it follows that
\[ \frac{ \nu (\ft (P_i) ) }{ |\ft(x_0) - \ft (y) | } \leq \frac{ |\ft(y) - \ft(z) |}{ |\ft(x_0) - \ft (y)| } \leq \eta \left ( \frac{|y-z|}{|x_0 - y| } \right )
\leq \eta \left ( \frac{ 2t\sqrt{n-1} }{t/2} \right ) = \eta (4\sqrt{n-1} ).\]
We conclude that
\[ \nu (\ft(P_i)) \leq \eta (4\sqrt{n-1} ) \operatorname{dist} (\ft(x_0) , \partial Q_i).\]
Thus for $s = \nu(\ft (P_i) ) / \eta (4\sqrt{n-1} )$, the ball $B(\ft(x_0) , s)$ does not meet $\ft(\partial Q_i)$. It follows that
\[ \vol_{n-1} \ft (Q_i) \geq \frac{ \Omega_{n-1} \nu (\ft (P_i)) ^{n-1} }{\eta (4\sqrt{n-1})^{n-1} },\]
from which the lemma follows.
\end{proof}

\section{Volume comparison}

\begin{lemma}
\label{lem:4}
Let $\ft :B_M \to B$ be $\eta$-quasisymmetric and let $t\in (0,1/2)$. Let $S_t = Q \times [t_0,t_0+t]$ and denote by $V_t$ the volume $\vol_n S_t$. Then there exist $C_4>1$, depending only on $n$ and $\eta$, and $t_1>0$ so that for $t \in (0,t_1)$ we have
\[ \frac{1}{C_4} \leq \frac{ \widetilde{\rho _f} (t_0+t) - \widetilde{\rho _f} (t_0) }{V_t} \leq C_4.\]
\end{lemma}

\begin{proof}
By the definition of $\widetilde{\rho_f}$, we have
\begin{equation}
\label{eq:l4eq1}
\widetilde{\rho_f}(t_0+t) - \widetilde{\rho_f}(t_0) = \log \rho_f ( e^{t_0+t}) - \log \rho_f(e^{t_0}) = \log \left ( \frac{ \rho_f(e^{t_0}e^t) }{\rho_f( e^{t_0} ) } \right ).
\end{equation}
Intepreting $\rho_f(s)$ as the normalized volume $\vol_n f(B(0,s)) / \Omega_n$, \eqref{eq:l4eq1} yields
\begin{align}
\label{eq:l4eq2}
\widetilde{\rho_f}(t_0+t) - \widetilde{\rho_f}(t_0) &= \log \left ( \frac{ \vol_n f(B(0,e^{t_0}e^t ) ) }{\vol_n f(B(0,e^{t_0})) } \right ) \nonumber\\
&= O \left ( \frac{ \vol_n f(B(0,e^{t_0}e^t ) ) - \vol_n f(B(0,e^{t_0}))} { \vol_n f(B(0,e^{t_0}))} \right ) 
\end{align}
as $t\to 0$. 
Denoting by $A_t$ the annulus $\{ x : e^{t_0} \leq |x| \leq e^{t_0}e^t \}$, we see that the numerator on the right hand side of \eqref{eq:l4eq2} is $\vol_n f(A_t)$. Moreover, $V_t$ is related to $A_t$ via $V_t = \vol_n \mathcal{Z}^{-1}( f(A_t))$.

Now, $\mathcal{Z}^{-1}$ is locally bi-Lipschitz away from the origin, but we want an estimate for $\mathcal{Z}^{-1}$ that will work for all small $t_0$. To that end, we observe that on a punctured neighbourhood of the origin, we have
\[ \mathcal{Z}^{-1}(x) = \mathcal{Z}^{-1} \left ( \frac{x}{\ell_f(e^{t_0}) } \right ) - \left [ \ln \frac{1}{\ell_f (e^{t_0} ) } \right ] e_n.\]
Since $f$ is quasiconformal, it has finite linear distortion at $0$, that is, there exist $r_0>0$ and a constant $H\geq 1$ such that
\[ \frac{ L_f(r) }{\ell_f(r) } \leq H \]
for $0<r<r_0$. Hence if $e^{t_0} < r_0$, we may conclude that the image of $f(A_t)$ under the map $x/\ell_f(e^{t_0})$ is contained in the ring $R_H = \{ x : 1\leq |x| \leq H \}$. Since this set is compact in $\R^n \setminus \{ 0 \}$, $\mathcal{Z}^{-1}$ is $\alpha$-bi-Lipschitz on $R_H$, for some $\alpha\geq 1$. Finally, the map $x - \left [ \ln \frac{1}{\ell_f (e^{t_0} ) } \right ] e_n$ is a translation, and hence just an isometry. Putting this together, we obtain
\begin{equation}
\label{eq:l4eq3} 
\frac{ \vol_n f(A_t) }{\alpha^n \ell_f (e^{t_0})^n } \leq  V_t \leq \frac{ \alpha^n \vol_n f(A_t) }{\ell_f(e^{t_0}) ^n }.
\end{equation}

By \eqref{eq:l4eq2} and \eqref{eq:l4eq3} we have
\begin{align*}
\widetilde{\rho_f}(t_0+t) - \widetilde{\rho_f}(t_0)  &= O \left ( \frac{ \vol_n f(A_t) }{\vol_n f(B(0,e^{t_0} )) } \right ) \\
&= O \left ( \frac{ \ell_f(e^{t_0})^n V_t }{\vol_n f(B(0,e^{t_0} )) } \right )\\
&= O(V_t),
\end{align*}
where we use the fact that $\ell_f(e^{t_0})^n$ is comparable to $\vol_n f(B(0,e^{t_0}))$ for all small enough $e^{t_0}$. This completes the proof.
\end{proof}

\section{Proof of Theorem \ref{thm:1}}

Assuming $\widetilde{\rho_f}$ is locally $L$-bi-Lipschitz at every $t\in (-\infty, M)$, for some $L$ depending on $n,K$ and $P$, it follows that $\widetilde{\rho_f}$ is globally $L$-bi-Lipschitz on $(-\infty, M)$. Hence it is enough to show that $\widetilde{\rho_f}$ is locally $L$-bi-Lipschitz.
Since the quasisymmetry function $\eta$ of $\ft$ depends only on the maximal dilatation $K$ of $f$ (and the maximal dilatation of the fixed map $\mathcal{Z}$), at any place below where a constant depends on $\eta$, it depends on $K$. 

In the following subsections, we will show that bounds exist for 
\[ \frac{ \widetilde{\rho_f} (t_0+t) - \widetilde{\rho_f}(t_0) }{t} \]
as $t \to 0$ for $t>0$. The arguments below can be easily modified to take into account the case where $t<0$, and so we focus solely on the case when $t>0$.

\subsection{The lower bound}

Suppose $\ft$ satisfies the hypotheses of Theorem \ref{thm:1}, $\epsilon >0$ is small and for some $t_0 < M-1$, some $0<t_2<1/2$ and all $0<t<t_2$, we have
\[ \frac{ \widetilde{\rho_f} (t_0+t) - \widetilde{\rho_f}(t_0) }{t} <\epsilon. \]
Our goal here is to show that if $\epsilon$ is too small, we obtain a contradiction.

Set $S_t = Q \times [t_0 , t_0+t]$ and let $V_t = \vol_n \ft (S_t)$. 
By Lemma \ref{lem:4} it follows that for $0<t<\min \{ t_1,t_2 \}$ we have
\begin{equation}
\label{eq:lb1}
V_t \leq C_4 \epsilon t.
\end{equation}

Next, cover $Q \times \{t_0 \}$ by $(n-1)$-dimensional cubes of side length $( \lfloor 1/t \rfloor )^{-1}$ and form the $n$-dimensional boxes $P_i$, for $i=1,\ldots, N$. Since the boxes $P_i$ cover $S_t$ with overlaps only on their boundaries, we have
\[ V_t = \sum_{i=1}^N \vol_n \ft (P_i ).\]
By Lemma \ref{lem:qsp}, this yields
\[ V_t \geq C_1^{-1} \sum_{i=1}^N \left ( \operatorname{diam} \ft (P_i) \right )^n.\]
Applying \eqref{eq:lp} we obtain
\begin{equation}
\label{eq:lb2}
V_t \geq C_1^{-1} N^{1/(1-n)} \left ( \sum_{i=1}^N \left ( \operatorname{diam} \ft (P_i) \right ) ^{n-1} \right ) ^{n/(n-1)}.
\end{equation}
Next, let $\mathcal{P} : B \to Q$ denote the orthogonal projection from the beam $B$ onto $Q$. It is clear that 
\[   \diam \mathcal{P}( \ft (P_i) )  \leq \diam \ft (P_i) \]
and hence that
\begin{equation}
\label{eq:lb3} 
\vol_{n-1} \mathcal{P}( \ft (P_i) ) \leq \frac{ \Omega_{n-1} (\diam \ft (P_i))^{n-1} }{2^{n-1} }.
\end{equation}
Now, we must have 
\[ Q \subset \bigcup_{i=1}^N \mathcal{P} ( \ft (P_i ) ),\]
for otherwise $\partial f(B(0,e^{t_0}))$ would not separate $0$ and $\infty$.
Hence 
\begin{equation}
\label{eq:lb4}
\sum_{i=1}^N \vol_{n-1} \mathcal{P}( \ft (P_i) ) \geq \vol_{n-1} Q = 2.
\end{equation}
Combining \eqref{eq:lb2}, \eqref{eq:lb3} and \eqref{eq:lb4}, we obtain
\begin{equation}
\label{eq:lb5}
V_t \geq C_1^{-1} N^{1/(1-n)} \left ( 2^n \Omega_{n-1}^{-1} \right ) ^{n/(n-1)} \geq C_5 N^{1/(1-n)},
\end{equation}
where $C_5$ depends only on $n$ and $\eta$.
Since $N/t^{1-n} \to 2$ as $t\to 0$, by \eqref{eq:lb1} and \eqref{eq:lb5} we have
\[ C_4 \epsilon t \geq 2^{1/(1-n)}C_5 t ,\]
which yields a contradiction if $\epsilon$ is small enough. 

\subsection{The upper bound}

Suppose $\ft$ satisfies the hypotheses of Theorem \ref{thm:1}, $\epsilon >0$ is small and for some $t_0 < M-1$, some $0<t_2<1/2$ and all $0<t<t_2$, we have
\[ \frac{ \widetilde{\rho_f} (t_0+t) - \widetilde{\rho_f}(t_0) }{t} >\frac{1}{\epsilon}. \]
Our goal here is to show that if $\epsilon$ is too small, we obtain a contradiction.

As with the lower bound, set $S_t = Q \times [t_0,t_0+t]$ and let $V_t = \vol_n \ft (S_t)$.
By Lemma \ref{lem:4} it follows that for $0<t<\min \{ t_1,t_2 \}$ we have
\begin{equation}
\label{eq:ub1}
V_t \geq \frac{t}{C_4\epsilon}.
\end{equation}

Next, cover $Q \times \{t_0 \}$ by $(n-1)$-dimensional cubes of side length $( \lfloor 1/t \rfloor )^{-1}$ and form the $n$-dimensional boxes $P_i$, for $i=1,\ldots, N$ with base $Q_i$. Since the boxes $P_i$ cover $S_t$ with overlaps only on their boundaries, we have
\[ V_t = \sum_{i=1}^N \vol_n \ft (P_i ).\]
By Lemma \ref{lem:qsp}, this yields
\[ V_t \leq C_1 \sum_{i=1}^N \left ( \operatorname{diam} \ft (P_i) \right )^n.\]
Then by Lemma \ref{lem:2}, we obtain
\[ V_t \leq C_1 C_2^n \sum_{i=1}^N \left ( \nu ( \ft (P_i) ) \right )^n .\]
Now applying Lemma \ref{lem:3}, we see that
\begin{equation}
\label{eq:ub2} 
V_t \leq C_1C_2^n C_3^{n/(n-1)} \sum_{i=1}^N \left ( \vol_{n-1} \ft (Q_i) \right )^{n/(n-1)} .
\end{equation}

Next, we use the fact that $f$ is a BIP map. Setting
\[ \gamma_t(x_1,\ldots, x_{n-1}) = \ft (x_1,\ldots, x_{n-1},t),\]
then for all $t<M$,
\[ \int_Q \left | \left | \Pi \left ( \frac{ \partial \gamma_t}{\partial x_1 } , \ldots, \frac{ \partial \gamma_t}{\partial x_{n-1} } \right ) \right | \right |^{n/(n-1)} _2 dV \leq P.\]
As a shorthand, for $y_i \in Q$ and $t<M$, we write $\Pi(y_i)$ for the vector
\[  \Pi \left ( \frac{ \partial \gamma_t}{\partial x_1 }(y_i) , \ldots, \frac{ \partial \gamma_t}{\partial x_{n-1} }(y_i) \right ) .\]

We may suppose that $t$ is small enough that the partition of $Q\times \{t_0 \}$ given by $Q_1,\ldots, Q_N$ yields
\begin{equation}
\label{eq:ub3} 
\vol_{n-1} \ft (Q_i)  \leq 2||\Pi(y_i)||_2 \vol_{n-1} (Q_i) \leq 2^{n-1} ||\Pi(y_i)||_2 t^{n-1},
\end{equation}
where $y_i \in Q_i$ is chosen so that all the partial derivatives of $\gamma_{t_0}$ exist at $y_i$, and the last inequality follows from \eqref{eq:qsp0}. We may also suppose that $t$ is chosen small enough that
\begin{equation}
\label{eq:ub4} 
\sum_{i=1}^N ||\Pi(y_i)||_2^{n/(n-1)} t^{n-1} \leq 2 \int_Q  \left | \left | \Pi \left ( \frac{ \partial \gamma_t}{\partial x_1 } , \ldots, \frac{ \partial \gamma_t}{\partial x_{n-1} } \right ) \right | \right |^{n/(n-1)} _2 dV.
\end{equation}

It now follows from \eqref{eq:ub2}, \eqref{eq:ub3} and \eqref{eq:ub4} that
\begin{align*} 
V_t &\leq C_1C_2^n C_3^{n/(n-1)} \sum_{i=1}^N \left ( \vol_{n-1} \ft (Q_i) \right )^{n/(n-1)}\\
&\leq 2^n C_1C_2^nC_3^{n/(n-1)} \sum_{i=1}^N\left (  ||\Pi(y_i)||_2^{n/(n-1)} t^{n-1} \right ) t \\
&\leq 2^{n+1}C_1C_2^nC_3^{n/(n-1)} \left ( \int_Q  \left | \left | \Pi \left ( \frac{ \partial \gamma_t}{\partial x_1 } , \ldots, \frac{ \partial \gamma_t}{\partial x_{n-1} } \right ) \right | \right |^{n/(n-1)} _2 dV \right ) t \\
& \leq C_6Pt,
\end{align*}
where $C_6$ depends only on $n$ and $K$.
Combining this with \eqref{eq:ub1}, we have
\[ \frac{t}{C_4\epsilon} \leq C_6Pt,\]
which yields a contradiction for $\epsilon$ small enough.

We emphasize that in both the upper and lower bounds the contradiction is obtained by finding $\epsilon$ to be too small relative to constants that depend only on $n$, $K$ and $P$. Hence $\widetilde{\rho_f}$ is locally $L$-bi-Lipschitz for some $L$ depending only on $n$, $K$ and $P$, which completes the proof of Theorem \ref{thm:1}.

\section{Non-rectifiable images}

\begin{proof}[Proof of Proposition \ref{prop:0}]

Here we show in dimension two that if the image of a cross-section of the beam is a non-rectifiable curve, then $\widetilde{\rho_f}$ is not bi-Lipschitz. Recall that $\ft$ is defined in $[0,2] \times (-\infty, M)$, so suppose that $\gamma (x):= \ft (x,t_0)$ parameterizes a non-rectifiable curve. In particular, its Hausdorff $1$-measure is infinite.

Using the notation established above, we have by Lemma \ref{lem:qsp} that
\[ V_t = \sum_{i=1}^N \vol_n \ft (P_i)  \geq \frac{1}{C_1} \sum _{i=1}^N \left ( \diam \ft (P_i) \right )^2.\]
Since $\gamma$ has infinite Hausdorff $1$-measure, it follows that given $\epsilon >0$ we can find $t_1>0$ so that if $0<t<t_1$, we have
\[ \sum_{i=1}^{N} \diam \ft (P_i) > \frac{1}{\epsilon} .\]
By \eqref{eq:lp} and \eqref{eq:qsp0} we then have
\[ V_t \geq \frac{1}{C_1 N} \left ( \sum_{i=1}^N \diam \ft (P_i) \right )^2 > \frac{t}{C_1 \epsilon } \]
for $0<t<t_1$. An application of Lemma \ref{lem:4} then shows that
\[  \frac{ \widetilde{\rho _f} (t_0+t) - \widetilde{\rho _f} (t_0) } {t} > \frac{1}{C_1C_4 \epsilon} \]
for $0<t<t_1$. Choosing $\epsilon $ small enough shows that $\widetilde{\rho_f}$ is not bi-Lipschitz.

\end{proof}

\section{Simple infinitesimal spaces}

\begin{proof}[Proof of Theorem \ref{thm:2}]
Recall the asymptotic representative
\[ D(x) = \rho_f(x) g(x/|x|), \]
where $g$ is the unique element in $T(x_0,f)$. By \cite[Proposition 4.18]{GMRV}, 
\begin{equation}
\label{eq:t2e1} 
g(x) = |x|^d g(x/|x|)
\end{equation}
for some $d>0$, and where $g:S^{n-1} \to \R^n$ is BLD.
In our case, $g$ is bijective onto its image, as $i(x_0,f)=1$, and so $g|_{S^{n-1}}$ is bi-Lipschitz. 
Since $g$ is quasiconformal, fixes $0$ and preserves the measure of the unit ball, there exists $C = C(K) \geq 1$ so that $g(S^{n-1})$ is contained in the ring $R = \{ x\in \R^n : 1/C \leq |x| \leq C\}$.
As $\mathcal{Z}$ is bi-Lipschitz on $Q \times \{ 0 \} $ and $\mathcal{Z}^{-1}$ is bi-Lipschitz on $R$, and in particular on $g(S^{n-1})$, it follows that 
\[ \widetilde{g} : Q\times \{0 \}\to B \]
is bi-Lipschitz.
Again letting $\mathcal{P} :B \to Q$ be the orthogonal projection, we see from \eqref{eq:t2e1} that $\widetilde{g} : B \to B$ is given by
\[ \widetilde{g} (x) = \widetilde{g}( \mathcal{P}(x)  ) + (0,\ldots, 0, d [x]_n  ).\]
Since the restriction of $\widetilde{g}$ to each slice $Q \times \{t \}$ is bi-Lipschitz, and since $\widetilde{g}$ is quasiconformal, it follows from the bounded linear distortion that $\widetilde{g}$ is itself bi-Lipschitz.

Now, the Zorich transform of $D$ is given by
\[ \widetilde{D} (x) = \widetilde{g} ( \mathcal{P}(x) ) + ( 0,\ldots, 0, \widetilde{\rho_f}([x]_n) ).\]
By Theorem \ref{thm:1}, $\widetilde{\rho_f}$ is bi-Lipschitz on $(-\infty, M)$. 
For every $t<M$, it follows that there exists $t' \in \R$ such that
\[ \widetilde{D} |_{ Q \times \{ t \} } = \widetilde{g} |_{Q \times \{ t' \} }.\]
Writing $h(t) = \widetilde{\rho_f}(t) / d$, the map $\widetilde{g}^{-1} \circ \widetilde{D}$ is nothing other than
\[ (x_1,\ldots, x_{n-1} , x_n ) \mapsto (x_1 , \ldots, x_{n-1} , h(x_n) ).\]
Since $h$ is bi-Lipschitz on $(-\infty, M)$, it follows that $\widetilde{g}^{-1} \circ \widetilde{D}$ is bi-Lipschitz on $B_M$. Since $\widetilde{g}$ is bi-Lipschitz, we conclude that $\widetilde{D}$ is also bi-Lipschitz and thus that $D$ itself is quasiconformal.
\end{proof}

\end{document}